\RequirePackage{fix-cm}
\documentclass[smallextended]{svjour3}       
\smartqed  
\usepackage{graphicx}
\usepackage{amssymb}
\begin{document}

\title{On abstract ${\cal F}$-systems\thanks{This work was partially supported by the National Agency for Scientific and Technological Promotions (ANPCYT) (grant PICT 2017-1702), and Universidad Nacional del Sur (grant PGI 24/I265), Argentina.}
}
\subtitle{A graph-theoretic model for paradoxes involving a falsity predicate and its application to argumentation frameworks}

\titlerunning{On abstract ${\cal F}$-systems}        

\author{Gustavo Bodanza      
}


\institute{Departamento de Humanidades, Universidad Nacional del Sur, and Instituto de Investigaciones Econ\'{o}micas y Sociales del Sur (IIESS), CONICET \at
              Bah\'{\i}a Blanca, Argentina \\
              Tel.: +54-291-4595138\\
              \email{bodanza@gmail.com}           
}

\date{Received: date / Accepted: date}

\maketitle

\begin{abstract}
${\cal F}$-systems are digraphs that enable to model sentences that predicate the falsity of other sentences. Paradoxes like the Liar and Yablo's can be analyzed with that tool to find graph-theoretic patterns. In this paper we present the ${\cal F}$-systems model abstracting from all the features of the language in which the represented sentences are expressed. All that is assumed is the existence of sentences and the binary relation `\ldots affirms the falsity of\ldots' among them. The possible existence of non-referential sentences is also considered. To model the sets of all the sentences that can jointly be valued as true we introduce the notion of \emph{conglomerate}, the existence of which guarantees the absence of paradox. Conglomerates also enable to characterize \emph{referential contradictions}, i.e. sentences that can only be false under a classical valuation due to the interactions with other sentences in the model. A Kripke's style fixed point characterization of groundedness is offered and fixed points which are complete (meaning that every sentence is deemed either true or false) and consistent (meaning that no sentence is deemed true and false) are put in correspondence with conglomerates. Furthermore, argumentation frameworks are special cases of ${\cal F}$-systems. We show the relation between  \emph{local conglomerates} and admissible sets of arguments and argue about the usefulness of the concept for argumentation theory.

\end{abstract}

\section{Introduction}
\label{intro}
Semantic paradoxes like the Liar and Yablo's involve sentences that predicate the falsity of other sentences. Cook \cite{cook_2004} introduced the novelty of using graph-theoretic tools for dealing with semantic paradoxes involving a falsity predicate. Rabern, Rabern and Macauley \cite{Rabern2013} coined the term `${\cal F}$-systems' to refer ``sentence systems which are restricted in such a way that all the sentences can only say that other sentences in the system are false''. The works by Cook and Rabern et al. concentrate on systems where \emph{every} sentence affirms the falsity of some other sentence(s). That is represented by serial or sink-free digraphs (roughly, every node ``shoots'' at least one arrow). On the other hand, Beringer and Schindler \cite{beringer_schindler_2017} also consider sentences that do not refer to other sentences. We will follow this last approach since it is more general and, particularly, enables to take into account the interaction among object language sentences (like `Snow is white') and metalanguage sentences (like `The sentence `Snow is white' is false'). However, we bring the analysis to a more abstract level where the specificities of the underlying languages  are (or the hierarchy of languages is) irrelevant to find the graph-theoretic patterns that characterize paradox. The aim is to get the simplest model that enables that. All we need is a set $S$ of nodes, which represent primitive entities we call \emph{sentences} (indeed, they can be understood as names of sentences of a given language) and a binary relation $F\subseteq S\times S$, i.e. a set of directed edges or arrows such that for any pair of sentences $x$ and $y$ of $S$, $(x, y)\in F$ is understood as `$x$ says that $y$ is false'. In this way, for example, we can model the relationship between the (English) sentences  `Snow is white' and  `The sentence `Snow is white' is false', through an ${\cal F}$-system $\cal F$ $=$ $\langle S, F\rangle$, where $S=\{a, b\}$ and $F=\{(b, a)\}$, and where $a$ and $b$ represent `Snow is white' and  `The sentence `Snow is white' is false', respectively.

Semantic paradoxes are sets of sentences to which it is not possible to assign a classical truth-value (true/false) to all of them at the same time. We will represent the assignment of truth-values through \emph{labellings}\footnote{We borrow the term `labelling' from \cite{caminada2006}. In \cite{beringer_schindler_2017}, the term `decoration' is used instead.}  on the nodes of the ${\cal F}$-systems, in such a way that every node can be labeled with $\tt T$ (for $\tt T$rue), $\tt F$ (for $\tt F$alse) or $\tt U$ (for $\tt U$ndetermined). Paradoxical ${\cal F}$-systems will be such that every labelling can only put the label $\tt U$ on some nodes.  ``Classical'' labellings, i.e. those that can put $\tt T$ or $\tt F$ on every sentence, will be put in correspondence with graph-theoretic patterns that we will call \emph{conglomerates}. 
The notion of conglomerate extends that of \emph{kernel} used by Cook, which represent a subset of sentences that can be true together. Kernels are suitable for capturing classic assignments of truth values in systems where each sentence refers to other sentences. But if we use this notion in systems that include sinks, kernels will only allow them to be represented as true. Conglomerates, although they will lead to similar formal results regarding paradoxes, will enable a more intuitive representation since object language sentences will be considered with any truth-value. 

Another aim of this work is to define a Kripke's style  fixed point operator to characterize groundedness in ${\cal F}$-systems \cite{Kripke1975}. Grounded sentences are, roughly, those which truth-value can be tracked through the reference path until a sentence with a definite truth-value (the ``ground''). We will define complete and consistent fixed points (meaning that every sentence is deemed either true or false and no sentence is deemed true and false) and show that they correspond exactly to conglomerates.

Conglomerates also enable to characterize \emph{referential contradictions (tautologies)}, i.e. sentences that can only be false (true) under a classical valuation, due to the interactions with other sentences in the model. This is an advantage with respect to kernels, which are unable to do that.

Finally, we define \emph{local conglomerates}, a notion that enables to cover and extend that of \emph{admissibility} in  Dung's argumentation frameworks.  As shown by Dyrkolbotn \cite{dyrkolbotn2012}, argumentation frameworks are special cases of ${\cal F}$-systems where arguments play the role of sentences and $F$ is interpreted as an attack relation. Admissibility formalizes the idea of sets of arguments that can be defended between them. Local conglomerates cover that idea and give it a twist: they deem ``admissible'' also sets of  arguments that can be defended together against any argument, except those that promote some non-preferred value, which suggests a new semantics for value-based argumentation frameworks \cite{bench-capon}.

The paper is organized as follows. In Section \ref{sec2} we define abstract ${\cal F}$-systems, labellings, and the notion of conglomerate. In Section \ref{ground} we give a fixed-point characterization of groundedness, and show the relation of conglomerates with complete and consistent fixed points. In Section \ref{refcon} we define referential contradictions and tautologies, and show that the transitivity of $F$ in non-paradoxical systems is a sufficient condition for their existence. Moreover, transitivity --as shown by Cook \cite{cook_2004}-- can also be a source of paradox as well as odd-length cycles. We comment on those points in Section \ref{suf}. In Section \ref{dung} we introduce the notion of local conglomerate and apply it to cover and extend that of admissibility in Dung's argumentation frameworks. Final conclusions are summarized in Section \ref{conc}.

\section{Abstract ${\cal F}$-systems}
\label{sec2}

\begin{definition} An (\emph{abstract}) ${\cal F}$-$system$ is a pair $\cal F$ $=$ $\langle S, F\rangle$ where $S$ is a set which elements are primitive entities called \emph{sentences} and $F\subseteq S\times S$ is a binary relation among sentences.
\end{definition}

\noindent For every $x\in S$, we define $\overrightarrow{F}(x)=\{y\in S: (x, y)\in F\}$ and $\overleftarrow{F}(x)=\{y\in S: (y, x)\in F\}$, and for every subset $A\subseteq S$, $\overrightarrow{F}(A)$ $=$ $\bigcup_{x\in A} \overrightarrow{F}(x)$ and $\overleftarrow{F}(A)$ $=$ $\bigcup_{x\in A} \overleftarrow{F}(x)$. If $\overrightarrow{F}(x)=\emptyset$, $x$ is said to be a \emph{sink}, and we define $sinks(A)$ $=$ $\{x\in A: x$ is a sink$\}$.  In order to avoid misrepresentations, we assume that non-sink sentences do not assert anything more than what is represented in $F$ (and, naturally, sink sentences do not assert anything about other sentences). To illustrate the kind of issues we want to avoid, consider the following example. Let ${\cal F}$ $=$ $(\{x, y\}, \{(y, x)\})$. Then we want to interpret that $x$ is true if $y$ is false and $x$ is false if $y$ is true; moreover, we want to interpret that if $x$ has an undetermined truth value then the value of $y$ is undetermined, too. However, if we accept the interpretation: $x$ $=$ `Snow is red' and $y$ $=$ `$x$ is false and the snow is blue', then the above considerations about the truth and falsity of $x$ and $y$ would not be valid, since $x$ and $y$ could both be false. Though it is true that $y$ affirms the falsity of $x$, the component `the snow is blue' of $y$ which is ``hidden'' in the representation can yield anomalous interpretations. The abstract level of the model does not allow to represent such molecular sentences since there are no elements to represent logical connectives. Hence, we leave that kind of interpretations out of the scope of the model. On the other hand, the only molecular sentences that can be represented in the model, preserving the intuitions about the assignment of truth values, are conjunctions of falsity assertions about other sentences like, for instance, `$x$ says that both $y$ and $z$ are false', which can be modeled as $\{(x, y), (x, z)\}\subseteq F$.\\

Since ${\cal F}$-systems define digraphs, we can see the assignment of truth values to the sentences as \emph{labels} on the nodes of a digraph. We consider three labels, $\tt{T}$, $\tt{F}$ and $\tt{U}$, for $true$, $false$ and $undetermined$, respectively. The non-classical value $undetermined$ is intended to express either that the actual value is unknown (as in the case of conjectures) or just the impossibility of assigning a classical truth value (as in the case of paradoxes).

\begin{definition}
Given ${\cal F}=\langle S, F\rangle$, a \emph{labelling} on ${\cal F}$ is a total function $L$ such that:
\begin{enumerate}
  \item $L:$ $S$ $\rightarrow$ $\{\tt{T}, \tt{F}, \tt{U}\}$, and
  \item for all $x\in S\setminus sinks(S)$
  \begin{enumerate}  
    \item $L(x)=\tt{F}$ iff $L(z)=\tt{T}$ for some $z\in \overrightarrow{F}(x)$, and
    \item $L(x)=\tt{T}$ iff $L(z)=\tt{F}$ for every $z\in \overrightarrow{F}(x)$.
	\end{enumerate}
\end{enumerate}
\end{definition}

\noindent Note that the assignment of values to sink nodes is unrestricted. Moreover, for every ${\cal F}$-system there always exist a labelling that labels all the nodes with $\tt{U}$.

\begin{definition}
A labelling $L$ on ${\cal F}$ is \emph{classical} iff for every $x\in S$, $L(x)\neq\tt{U}$. 
\end{definition}

\noindent Paradoxes in ${\cal F}$-systems can be characterized as follows:

\begin{definition}
An ${\cal F}$-system is \emph{paradoxical} iff it has no classical labellings. Moreover, a sentence $x$ is paradoxical iff $L(x)=\tt{U}$ for every labelling $L$. 
\end{definition}

\begin{example}
Let ${\cal F}=\langle\{a_{k}\}_{_{k\in\mathbb{N}}}, \{(a_{k}, a_{m})\}_{_{k<m}}\rangle$ (representing the Yablo's paradox). Then ${\cal F}$ does not have any classical labelling.
\end{example}

Classical labellings determine partitions of the set of sentences in which there are two subsets, representing the true and the false sentences, respectively. If the system is paradoxical then we cannot have any such partition. Another way of characterizing the subsets of true and false sentences is by defining suitable properties. We introduce here the notion of \emph{conglomerate}.

\begin{definition} Given ${\cal F}$ $=$ $\langle S, F\rangle$, a \emph{conglomerate} is a 
subset $A\subseteq S$ that satisfies:
\begin{enumerate}
  \item\label{se2} $\overleftarrow{F}(A)\subseteq S\setminus A$, and
  \item\label{se3} $(S\setminus A)\setminus sinks(S)\subseteq \overleftarrow{F}(A)$   
\end{enumerate}
\end{definition}

The idea is that a conglomerate coalesce all the sentences that can share the true value, leaving outside all and only the sentences that can share the false value. So, conglomerates can only exist in systems which sentences can be ``polarized'' into true and false. A conglomerate can also be understood in a Kripkean way as defining the extension of the truth predicate of the underlying language, being its complement in the system the anti-extension (we will return to this point in Section \ref{ground}). If a conglomerate exists, then we can say that the truth predicate is completely defined in the system. Since a conglomerate $A$ is supposed to comprise all the true sentences, the condition 1 says that it cannot contain two sentences such that one affirms the falsity of the other (i.e. $A$ is \emph{independent}). And $S\setminus A$ is supposed to comprise all the false sentences, so the condition 2 says that every non-sink sentence must assert the falsity of at least one true sentence (i.e. $A$ \emph{absorbs} every external non-sink node). This is different from kernels, which absorb every outer node. This implies that kernels comprise all the sinks, so these can only be interpreted as true sentences in such a model. In this sense, conglomerates seem to be more suitable than kernels to represent the Kripkean view: sinks may or may not belong to the conglomerates, representing object language sentences that may or may not be true. Furthermore, the notion of conglomerate clearly also encompasses that of kernel.

\begin{example}\label{ab}
Let ${\cal F}=\langle\{a, b\}, \{(b, a)\}\rangle$. Assume that ${\cal F}$ represents the relation between $a$: `I am wearing a hat' and $b$: `The sentence `I am wearing a hat' is false'. ${\cal F}$ has only one kernel, $\{a\}$, but it has two conglomerates, $\{a\}$, deeming `I am wearing a hat' as true and `The sentence `I am wearing a hat' is false' as false, and $\{b\}$, deeming `I am wearing a hat' as false and `The sentence `I am wearing a hat' is false' as true. Every kernel is a conglomerate, but not vice versa.
\end{example}


The notion of conglomerate is not well-defined, in the sense that some ${\cal F}$-systems have no conglomerates. As expected, those systems are the paradoxical ones.

\begin{example}
Let ${\cal F}=\langle\{a\}, \{(a, a)\}\rangle$ (the Liar paradox). Then ${\cal F}$ does not have any conglomerate.
\end{example}

\noindent The correspondence between conglomerates and classical labellings is easy to prove:

\begin{theorem}\label{theo}
$L$ is a classical labelling iff $A=\{x: L(x)=\tt{T}\}$ is a conglomerate.
\end{theorem}

\begin{proof} Let ${\cal F}=\langle S, F\rangle$.\\
\noindent (If) Let $A$ be a conglomerate of ${\cal F}$. Let $L$ be such that $\forall x (x\in A\rightarrow L(x)=\tt{T})$ and $\forall x (x\in S\setminus A\rightarrow L(x)=\tt{F})$. Then $L$ trivially satisfies the conditions of a classical labelling.\\
\noindent (Only if) Let $L$ be a classical labelling of ${\cal F}$ and let $A$  $=$ $\{x: L(x)=\tt{T}\}$ and $B$  $=$ $\{x: L(x)=\tt{F}\}$. (i) By definition, if $L(x)=\tt{T}$ then for all $z$ such that $z\in \overrightarrow{F}(x)$, $L(z)=\tt{F}$. Hence, by construction, $x\in A$ and $z\in B$. Then, for all $x,z\in A$, $z\notin \overrightarrow{F}(x)$. (ii) By definition, for all $x\in S$, if $L(x)=\tt{F}$ and $x$ is not a sink, then there exists some $z\in \overrightarrow{F}(x)$ such that $L(z)=\tt{T}$. Hence, by hypothesis, $x\in B$ and $z\in A$. Therefore, given (i) and (ii) we have that $A$ is a conglomerate.
\end{proof}

\begin{corollary}
${\cal F}$ is paradoxical iff it does not have any classical labelling.
\end{corollary}

\section{Groundedness}\label{ground}

The truth value of sentences asserting the falsity of other sentences depends on the truth value of the referred sentences. If the truth value of a sentence does not depend on that of other sentences, ``so that the truth value of the original statement can be ascertained, we call the original sentence \emph{grounded}, otherwise \emph{ungrounded}'' (Kripke, 1975: 694). In our framework, sentences at sink nodes (for instance, object language sentences) do not depend on other sentences in that sense, so their truth value depend on material (contingencies) or formal (tautologies or contradictions) facts that are exogenous to the model.  Taking as grounded all the sinks that are either true or false, the groundedness of all the remaining sentences of an ${\cal F}$-system will be determined in an iterated process very similar to Kripke's. In the base case, all the sinks determined as true belong to a set $S^{+}_{0}$  and all those determined as false belong to a set $S^{-}_{0}$.  That is, the \emph{partial set} $(S^{+}_{0}, S^{+}_{0})$ models the interpretation of the sink sentences. The systems considered by Cook and Rabern et al. are sink-free, hence no sentence is grounded in the above sense in those systems. Beringer and Schindler, on the other hand, consider the existence of sinks (representing true arithmetical sentences) and so their system constitute a special instance of an abstract ${\cal F}$-system in which groundedness can be tracked by dependence on the sinks.

\begin{definition}
Given ${\cal F}$ $=$ $\langle S, F\rangle$, a pair $(S^{+}_{0}$, $S^{-}_{0})$ is a \emph{ground base} iff $S^{+}_{0}\cup S^{-}_{0}= sinks(S)$ and $S^{+}_{0}$ $\cap$ $S^{-}_{0}$ $=$ $\emptyset$.
\end{definition}


Then, we can find the other grounded sentences by iterated applications of the following operator:

\begin{definition}
Given two subsets $S^{+}, S^{-}\subseteq S$, we define $\phi((S^{+}, S^{-}))$ $=$ $(S'^{+}, S'^{-})$, where
\begin{itemize}
\item[ ]  $S'^{+}=sinks(S^{+})\cup \{x: \emptyset\neq\overrightarrow{F}(x)\subseteq S^{-}\}$, and
\item[ ]  $S'^{-}=sinks(S^{-})\cup \{x: \emptyset\neq\overrightarrow{F}(x)\cap S^{+}\}$.
\end{itemize}
\end{definition}

\noindent That is, $S'^{+}$ includes the sinks that are already known as true plus all the sentences that only affirm the falsity of sentences already known as false; and $S'^{-}$ includes the sinks that are already known as false  plus all the  sentences that affirm the falsity of some sentence already known as true. So, starting from any ground base $(S^{+}_{0}$, $S^{-}_{0})$, iterated applications of $\phi$ will lead to a fixed point. A fixed point is any pair $(S^{+}, S^{-})$ $=$ $\phi((S^{+}, S^{-}))$. The fixed point $(S^{+}, S^{-})$ $=$ $\phi^{\infty}((S^{+}_{0}$, $S^{-}_{0}))$ reached by the above mentioned iteration procedure is the least one relative to the ground base $(S^{+}_{0}$, $S^{-}_{0})$, in the sense that any other fixed point $(S'^{+}, S'^{-})$ where $S^{+}_{0}\subseteq S'^{+}$ and $S^{-}_{0}\subseteq S'^{-}$ is such that $S^{+}\subseteq S'^{+}$ and $S^{-}\subseteq S'^{-}$. The operator $\phi$ is monotone. Let $(S^{+}, S^{-})\leq (S'^{+}, S'^{-})$ iff $S^{+}\subseteq S'^{+}$ and $S^{-}\subseteq S'^{-}$. Then:

\begin{remark}
(Monotony) If $(S^{+}, S^{-})$ $\leq$ $(S'^{+}, S'^{-})$ then $\phi((S^{+},$ $S^{-}))$ $\leq$ $\phi((S'^{+},$ $S'^{-}))$.
\end{remark}

\begin{proof}
Assume $(S^{+}, S^{-})\leq (S'^{+}, S'^{-})$. Let $\phi((S^{+},$ $S^{-}))$ $=$ $(T^{+},$ $T^{-})$ and $\phi((S'^{+}, S'^{-}))$ $=$ $(T'^{+}, T'^{-})$. Assume now $w\in T^{+}$ and $z\in T^{-}$. We have to prove that 1) $w\in T'^{+}$ and 2) $z\in T'^{-}$.\\
1) If $w\in S^{+}$ the result is obvious. Assume now $w\not\in S^{+}$. Then $w\in \{x: x\not\in S_{0}$ and $\forall y(y\in \overrightarrow{F}(x)\rightarrow y\in S^{-})\}$. Then, $w\not\in S_{0}$ and $\forall y(y\in \overrightarrow{F}(w)\rightarrow y\in T^{-})$. Therefore, $w\in T'^{+}$.\\
2) If $z\in S^{-}$ the result is obvious. Assume now $z\not\in S^{-}$. Then $z\in \{x: x\not\in S_{0}$ and $\exists y(y\in \overrightarrow{F}(x)\wedge y\in S^{+})\}$. Then, $z\not\in S_{0}$ and $\exists y(y\in \overrightarrow{F}(z)\wedge y\in T^{+})$. Therefore, $z\in T'^{-}$.
\end{proof}

\noindent The existence of the least fixed point is guaranteed by the monotony of $\phi$ and the fact that all the fixed points (relative to the same ground base) form a complete lattice (by Tarski's fixed points theorem).

Now, some ground bases can deem all the sentences grounded and others not. Consider the following example:

\begin{example}\label{relgr}
Let ${\cal F}$ $=$ $\langle\{a, b, c\}, \{(b, a), (b, c), (c, b)\}\rangle$. The only sink $a$ determines two possible ground bases: (1) $(\{a\}, \emptyset)$, and (2) $(\emptyset, \{a\})$. For (1), all $a$, $b$, and $c$ are grounded (true, false, and true, respectively), and for (2), only $a$ is grounded (false). By way of illustration, the least fixed point in each case is reached as follows:\\

\begin{tabular}{l l}
(1) 	& $\phi((\{a\}, \emptyset))$$=$$(\{a\}, \{b\})$ \\ 
	& $\phi^{2}((\{a\}, \emptyset))$$=$$(\{a, c\}, \{b\})$ \\  
	& \vdots \\
	& $\phi^{\infty}((\{a\}, \emptyset))$$=$$(\{a, c\}, \{b\})$\\ \\

(2) 	& $\phi((\emptyset, \{a\}))$$=$$(\emptyset, \{a\})$ \\ 
	& \vdots \\
	& $\phi^{\infty}((\emptyset, \{a\}))$$=$$(\emptyset, \{a\})$
\end{tabular}
\end{example}

\noindent Since an ${\cal F}$-system can have different ground bases, some of them can deem all the sentences grounded but others not. The following notion of groundedness takes into account those possibilities:

\begin{definition}
${\cal F}$ $=$ $\langle S, F\rangle$ is (\emph{relatively}) \emph{grounded} iff for every $x\in S$ and for (some) every ground base $(S^{+}_{0}, S^{-}_{0})$, $x\in S^{+}$ $\cup$ $S^{-}$, where $(S^{+}, S^{-})$ $=$ $\phi^{\infty}((S^{+}_{0}$, $S^{-}_{0}))$.
\end{definition}

The system in Example \ref{relgr} is relatively grounded but not grounded. On the other hand, ${\cal F}$-systems representing paradoxes as the Liar or Yablo's are neither grounded nor relatively grounded, as expected. Other systems are not grounded nor relatively grounded either, though they are not paradoxical. For instance,  ${\cal F}$ $=$ $\langle\{a, b\},$ $\{(a, b),$ $(b, a)\}\rangle$. Since there are no sinks, the only ground base is $(\emptyset, \emptyset)$, and $\phi^{\infty}((\emptyset, \emptyset))$ $=$ $(\emptyset, \emptyset)$ is the least fixed point. However, other fixed points like $(\{a\}, \{b\})$ and $(\{b\}, \{a\})$ imply that $a$ and $b$ can be consistently assigned a classical truth value.  This makes the difference with paradoxical systems, where no fixed point represents a consistent assignment of truth values to all the sentences.\\

The conglomerates of an ${\cal F}$-system can be characterized by means of the fixed points of $\phi$ as follows:

\begin{definition}
We say that $(S^{+}, S^{-})$ is \emph{complete} iff for every $x\in S$, $x\in S^{+}$ $\cup$ $S^{-}$, and \emph{consistent} iff $S^{+}$ $\cap$ $S^{-}$ $=$ $\emptyset$.
\end{definition}

\begin{theorem}\label{theo2}
$A$ is a conglomerate iff  $(A, S\setminus A)$ is a complete and consistent fixed point of $\phi$.
\end{theorem}

\begin{proof}
(Only if) Let $A$ be a conglomerate of ${\cal F}$ $=$ $\langle S, F\rangle$. Let $(S^{+}, S^{-})$ $=$ $\phi((A, S\setminus A))$. Then we have:\\
(i) $S^{-}$ $=$ $sinks(S\setminus A)\cup \{x: \exists y (y\in \overrightarrow{F}(x) \wedge y\in A)\}$, and\\
(ii) $S^{+}$ $=$ $sinks(A)\cup \{x: x\not\in sinks(S)$ and $\forall y(y\in \overrightarrow{F}(x)\rightarrow\ y\in S\setminus A)\}$.\\
To see that $(A, S\setminus A)$ is a fixed point of $\phi$ we have to prove that $S^{+}=A$ and $S^{-}= S\setminus A$. From (i), it follows that $S^{-}= S\setminus A$, since $A$ absorbs every non-sink node and, obviously, every sink of $S\setminus A$ belongs to $S\setminus A$. Let us prove now that $S^{+}=A$. (1) $S^{+}\subseteq A$: By \emph{reductio}, assume that there exists $x\not\in A$ such that $x$ is not a sink and $\forall y(y\in \overrightarrow{F}(x)\rightarrow\ y\in S\setminus A)$. But then $x$ is not absorbed by $A$, which contradicts that $A$ is a conglomerate. (2) $A\subseteq S^{+}$: Of course, $sinks(A)\subseteq S^{+}$. Let now $x\in A$ be a non-sink node. From the definition of conglomerate, $\overleftarrow{F}(A)\subseteq S\setminus A$. Therefore, it is clear that $\forall y(y\in \overrightarrow{F}(x)\rightarrow\ y\in S\setminus A)$. Finally, the fact that $(A, S\setminus A)$ is a complete and consistent fixed point is obvious.\\
(If) Let  $(S^{+}, S^{-})$ be a complete and consistent fixed point of ${\cal F}$ $=$ $\langle S, F\rangle$. By way of the absurd, assume $z\in \overrightarrow{F}(x)$ and $x,z\in S^{+}$. Then $\phi((S^{+}, S^{-}))$ $=$ $(S'^{+}, S'^{-})$ is such that $x\in S'^{+}$ and $x\in S'^{-}$. But this contradicts the consistency property. Therefore, (i) for all $x,z\in S^{+}$, $z\notin \overrightarrow{F}(x)$. Now, since $(S^{+}, S^{-})$ is complete and consistent it follows that $S^{-}=S\setminus S^{+}$. By this and because $(S^{+}, S^{-})$ is a fixed point of $\phi$, we have that $S\setminus S^{+}$ $=$ $sinks(S^{-})$ $\cup$ $\{x: \exists y(y\in \overrightarrow{F}(x)\wedge y\in S^{+})\}$, which in turn implies that (ii) $(S\setminus S^{+})\setminus sinks(S)$ $=$ $\{x: \exists y(y\in \overrightarrow{F}(x)\wedge y\in S^{+})\}$ $\subseteq$ $\overleftarrow{F}(A)$. Therefore, from (i) and (ii) and by definition, $S^{+}$ is a conglomerate.
\end{proof}

\begin{example}\label{relgr2} (Continuation of Example \ref{relgr})
There are three conglomerates, $\{a, c\}$, $\{b\}$, and $\{c\}$, that can be put in correspondence with the fixed points $(\{a, c\},$ $\{b\})$, $(\{b\},$ $\{a, c\})$, and $(\{c\}, \{a, b\})$, respectively. 
\end{example}

\section{Sentences with intrinsic value: referential contradictions and tautologies}\label{refcon}

We have seen that some ${\cal F}$-systems, by their own structural nature, have sentences that can only be labeled as undetermined ($\tt{U}$): those just characterized as paradoxical. In addition, some systems have sentences that can only be labeled as true ($\tt{T}$) and others as false ($\tt{F}$) by every classical labelling. We will say that those sentences are \emph{referential tautologies}/\emph{contradictions},  meaning that their truth/falsity is due to structural conditions of the system. 

\begin{definition}
Given ${\cal F}$ $=$ $\langle S, F\rangle$, $x\in S$ is a \emph{referential contradiction} (\emph{tautology}) iff $L(x)=\tt{F}$ ($L(x)=\tt{T}$), for every classical labelling $L$.
\end{definition}

As a consequence, referential contradictions and tautologies are related to conglomerates in the following way:

\begin{proposition}
Let ${\cal F}$ $=$ $\langle S, F\rangle$ be non-paradoxical. Then $x\in S$ is a referential contradiction (tautology) iff for every conglomerate $A$, $x\notin A$ ($x\in A)$.
\end{proposition}

\begin{proof}
It follows immediately from Theorem \ref{theo}.
\end{proof}




\noindent Referential tautologies are just sentences that assert the falsity of a contradiction. However, referential contradictions can exist independently of referential tautologies.

\begin{proposition}
Given ${\cal F}$ $=$ $\langle S, F\rangle$, if $x\in S$ is a referential tautology, then there exists some $z\in S$ such that $z$ is a referential contradiction (and $z\in \overrightarrow{F}(x)$).
\end{proposition}

\begin{proof}
By definition, if $x$ is a referential tautology then $L(x)=\tt{T}$ for every classical labelling $L$. By definition of labelling, $x$ cannot be a sink (otherwise it could be labeled with $\tt{T}$ by some classical labelling). Hence, there exists $z\in \overrightarrow{F}(x)$ and, obviously, $L(z)=\tt{F}$ in every classical labelling $L$.
\end{proof}

Note that referential contradictions and tautologies are not related to kernels in the same way as to conglomerates.

\begin{example}\label{ex7}
Let ${\cal F}$ $=$ $\langle\{a, b, c\}, \{(a, b), (b, c), (a, c)\}\rangle$. There exist two conglomerates, $\{c\}$ and $\{b\}$, deeming $a$, which is not paradoxical, as a referential contradiction. The only kernel $\{c\}$ is useless to capture that fact.
\end{example}

The previous example also shows that transitivity is a source of referential contradictions when the antecedent conditions of the transitive property are met in non-paradoxical systems.

\begin{proposition}
Let ${\cal F}$ $=$ $\langle S, F\rangle$ be non-paradoxical and $F$ be transitive. If $x, y, z\in S$ are such that $(x, y),$ $(y, z)\in F$, then $x$ is a referential contradiction. 
\end{proposition}

\begin{proof}
Assume the antecedent of the claim.  By the transitivity of $F$, $(x, z)\in F$. Let $L$ be a classical labelling of ${\cal F}$ (which exists due to the non-paradoxicality of ${\cal F}$). Then either (i) $L(z)=\tt T$ or (ii) $L(z)=\tt F$. If (i) is the case then $L(x)=\tt F$.  If (ii) is the case then either (a) $L(y)=\tt T$ or (b) $L(y)=\tt F$. If (a) is the case then $L(x)=\tt F$, and if (b) is the case then there exists $w\in S$ such that $(y, w)\in F$ and $L(w)=\tt T$. But then, by transitivity, $(x, w)\in F$. That implies that $L(x)=\tt F$. Hence, in any case $L(x)=\tt F$. Therefore, $x$ is a referential contradiction.
\end{proof}

In addition, if transitivity is satisfied by systems where every sentence refers to other sentences we get paradox, as we will see in the next section.

\section{Sufficient conditions for paradox}\label{suf}

Rabern et al. (\cite{Rabern2013}) identify structural properties of the digraphs as necessary conditions for paradox. The conditions are, basically, the existence of directed cycles (as in the case of the Liar) or double paths\footnote{A double path is a graph consisting of two non-trivial paths, both with common origin and end.} (as in the case of the Yablo's paradox). However, those conditions are not sufficient. In this section we show some sufficient conditions present in the literature and establish another one related to odd-length cycles.

\subsection{Transitivity}

$F$ is transitive iff for all $x, y, z\in S$, if $y\in \overrightarrow{F}(x)$ and $z\in \overrightarrow{F}(y)$ then $z\in \overrightarrow{F}(x)$.  For example, $F$ is transitive in Example \ref{ex7}. In that system, note that $a$ cannot be labelled with $\tt{T}$, but it can be labeled with $\tt{F}$ whenever $c$ is labeled with $\tt{T}$ or $\tt{F}$ (i.e. $a$ is a referential contradiction). Moreover, if \emph{every} sentence refers to the falsity of other sentence, then transitivity will also prevent the assignment of the $\tt F$ label. This result was showed by Cook \cite{cook_2004}.

\begin{definition}
${\cal F}=\langle S, F\rangle$ is \emph{unlimited transitive} iff (i) $F$ is transitive and (ii) ${\cal F}$ is a serial digraph (i.e. no node is a sink).
\end{definition}

\begin{proposition}
(Cook) If ${\cal F}$ is unlimited transitive then it is paradoxical.
\end{proposition}

\begin{proof}
Assume ${\cal F}=\langle S, F\rangle$ is not paradoxical. Then it has a classical labelling $L$. Then, for every $x\in S$, either $L(x)=\tt T$ or $L(x)=\tt F$. Assume $L(x)=\tt T$. Then, for all $y\in\overrightarrow{F}(x)$, $L(y)=\tt F$. Let now $y\in\overrightarrow{F}(x)$. Then, for some $z\in\overrightarrow{F}(y)$, $L(z)=\tt T$. But, by transitivity, $z\in\overrightarrow{F}(x)$, which implies that $L(z)=\tt F$. Contradiction. Assume now $L(x)=\tt F$. Then, for some $y\in\overrightarrow{F}(x)$, $L(y)=\tt T$. Then we can apply on $y$ the same argument as before to get a contradiction. Therefore, ${\cal F}$ is paradoxical.
\end{proof}

\noindent Both the Liar and Yablo's paradoxes can be modeled as unlimited transitive ${\cal F}$-systems.

\subsection{Odd-length cycles}

Unlike the Yablo's paradox, the Liar paradox contains a referential cycle. But not every referential cycle leads to paradox. Next we  define some conditions involving cycles that suffice to yield paradox.

\begin{definition}
Given ${\cal F}$ $=$ $\langle S, F\rangle$, the subset $O\subseteq S$ is an \emph{odd core} of ${\cal F}$ iff for some $n\geq 0$ there exist sentences $x_{1},\ldots, x_{2n+1}$ $\in$ $S$ such that $O$ $=$ $\{x_{1},\ldots, x_{2n+1}\}$, $\{(x_{i}, x_{i+1}): 1\leq i\leq 2n\}\cup\{(x_{2n+1}, x_{1})\}$ $\subseteq$ $F$, and for all $x\in O$, $|F(x)|= 1$. Moreover, ${\cal F}$ is \emph{odd} iff it has an odd core. 
\end{definition}

\noindent Informally, the definition says that an odd ${\cal F}$-system is such that there exists an odd-length cycle in $F$, the nodes of which can shoot exactly one arrow each (no matter how many arrows point to them).

\begin{proposition}
If ${\cal F}$ is odd then it is paradoxical.
\end{proposition}

\begin{proof}
Let ${\cal F}$ $=$ $\langle S, F\rangle$ be odd and let $O$ $=$ $\{x_{1},\ldots, x_{2n+1}\}$ be an odd core of ${\cal F}$. By way of contradiction, let us assume that $A$ is a conglomerate of ${\cal F}$. Assume, without lost of generality, that $x_{1}\in A$. Then, since for all $x\in O$, $|F(x)|= 1$, by the  absorption condition we have  $\{x_{1}, x_{3}, \ldots, x_{2n+1}\}\subseteq A$. But, since $(x_{2n+1}, x_{1})\in F$, that contradicts the independence property of $A$. Hence, we should have $x_{1}\in S\setminus A$. This implies that $\{x_{2}, x_{4}, \ldots, x_{2n}, x_{1}\}\subseteq A$. But, since $(x_{1}, x_{2})\in F$, we get again a contradiction with the independence property. Therefore, ${\cal F}$ cannot have any conglomerate, which means that it is paradoxical.
\end{proof}

As a consequence, we can see that the source of paradox in the Liar paradox is twofold: the ${\cal F}$-system is both unlimited transitive and odd. Yablo's paradox, in turn, only suffers  from unlimited transitivity.

Finally, it is worth mentioning that Dyrkolbotn \cite{dyrkolbotn2012} resumes some known sufficient conditions to avoid paradox in presence of odd-length cycles. In our framework, that result can be expressed as follows:

\begin{proposition}
(Dyrkolbotn) Any ${\cal F}$-system has a conglomerate if every odd-length cycle has one of the following
\begin{enumerate}
  \item at least two symmetric nodes
  \item at least two crossing consecutive chords (a \emph{chord} is an arrow on a cycle connecting two non-consecutive nodes)
  \item at least two chords with consecutive targets.
\end{enumerate}
\end{proposition}

\section{Local conglomerates and their relation with admissible sets in Dung's argumentation frameworks}\label{dung}

Conglomerates capture all the sentences that can be true together in systems that are free of paradoxes. In terms of Dyrkolbotn \cite{dyrkolbotn2012}, the absorption condition on conglomerates is \emph{global}, and that inhibit some systems of having conglomerates. But we can still want to know what sentences can be true together in systems containing paradoxes, even if that class is empty. That can be done by defining a \emph{local} version of the absorption condition. To that aim, in this section we define the notion of \emph{local conglomerate}.

P.M. Dung \cite{dung1995} has defined \emph{argumentation frameworks}, which are just special cases of ${\cal F}$-systems where $S$ is interpreted as a set of \emph{arguments} and $F$ as an \emph{attack} relation. The notion of \emph{admissibility} captures the idea of sets of arguments that can be defended jointly: $A\subseteq S$ is \emph{admissible} iff 1) $\overrightarrow{F}(A)\subseteq S\setminus A$ (i.e. the arguments that $A$ attacks are outside $A$), and  2) $\overleftarrow{F}(A)\subseteq\overrightarrow{F}(A)$ (i.e. $A$ attacks all its attackers).  In graph theory, and replacing $F$ with $F^{-1}$, that notion is known with the name of \emph{semikernel} \cite{GaleanaLara}. Following the same motivation that led us to replace the notion of kernel with that of conglomerate, we are going to replace the notion of semikernel with the following:

\begin{definition}
Given ${\cal F}$ $=$ $\langle S, F\rangle$, a \emph{local conglomerate} is a subset $A\subseteq S$ that satisfies:
\begin{enumerate}
  \item $\overleftarrow{F}(A)\subseteq S\setminus A$, and
  \item $\overrightarrow{F}(A)\setminus sinks(S)\subseteq\overleftarrow{F}(A)$.
\end{enumerate}
\end{definition}

It is easy to see that conglomerates are also local conglomerates, which enables to establish a hierarchy among the notions, including kernels. Local conglomerates only absorb those non-sink nodes to which they point at. As a consequence, the empty set of nodes is always a local conglomerate, which implies that the notion is well-defined: every ${\mathcal F}$-system has at least one local conglomerate. Moreover, due to the fact that local conglomerates are local ``absorbers'', they capture all the sentences that can have a classical truth value in contexts where paradoxes can be isolated.

\begin{example}
Let  ${\cal F}$ $=$ $\langle S, F\rangle$ $=$ $\langle\{a, b, c, d\}$, $\{(a, b),$ $(b, c),$ $(a, c),$ $(d, d)\}\rangle$. Then ${\cal F}$ has no conglomerates but has two non-empty local conglomerates, $\{b\}$ and $\{c\}$. They can be matched with the fixed points  $(\{b\}, \{a, c\})$ and $(\{c\}, \{a, b\})$, respectively.   Clearly, the paradoxical sentence $d$, which is not connected to the other sentences, is not present in those fixed points.
\end{example}

In Dung's argumentation frameworks, maximally (w.r.t. $\subseteq$) admissible sets of arguments are called \emph{preferred extensions}. Analogously, we can think of the maximal (w.r.t. $\subseteq$) local conglomerates of an ${\cal F}$-system, i.e. subsets of sentences that can all be true together at the same time and, if so, the remaining sentences of the system must necessarily be untrue. These can be put in correspondence with consistent fixed points that are maximal w.r.t. $\leq$.

\begin{theorem}\label{theo3}
$A$ is a maximal (w.r.t. $\subseteq$) local conglomerate iff $(A, \overrightarrow{F}(A)\cup \overleftarrow{F}(A))$ is a consistent maximal (w.r.t. $\leq$) fixed point of $\phi$.
\end{theorem}

\begin{proof}
(If) Let $(A, B)$ be a consistent maximal (w.r.t. $\leq$) fixed point of $\phi$. Then\\
(i) $\overleftarrow{F}(A)\subseteq S\setminus A$. To prove this, assume $x\in \overleftarrow{F}(A)$. By way of contradiction, assume now that $x\in A$ and let $z\in A$ be such that $z\in \overrightarrow{F}(x)$. Since $(A, B)$ is a fixed point, we have that $B$ $=$ $sinks(B)\cup\{x: \emptyset\neq \overrightarrow{F}(x)\cap A\}$. Then, $z\in B$. So, $z\in A\cap B$, which contradicts the consistency of $(A, B)$. Therefore, $x\not\in A$, which implies that $x\in S\setminus A$.\\
(ii) $\overrightarrow{F}(A)\setminus sinks(S)\subseteq \overleftarrow{F}(A)$. To prove this, let $u\in A$ and $w\in \overrightarrow{F}(u)\setminus sinks(S)$ (i.e. $w\in \overrightarrow{F}(A)\setminus sinks(S)$). Since $(A, B)$ is a fixed point, the hypothesis $w\not\in \overleftarrow{F}(A)$ implies that $w\not\in B$, since $B$ $=$ $sinks(B)\cup\{x: \emptyset\neq \overrightarrow{F}(x)\cap A\}$. But then $u\not\in A$, since $A$ $=$ $sinks(A)\cup\{x: \emptyset\neq \overrightarrow{F}(x)\subseteq B\}$. Contradiction. Therefore, $w\in \overleftarrow{F}(A)$.\\
From (i) and (ii) it follows that $A$ is a local conglomerate. Assume now that $A$ is not maximal, i.e. there exists some local conglomerate $A'$, such that $A\subset A'$.  Let $x\in A'\setminus A$. If $x$ is a sink then $x\in$ $sinks(A')\cup\{x: \emptyset\neq \overrightarrow{F}(x)\subseteq B\}$ $=$ $A'$. But this contradicts that $(A, B)$ is a maximal fixed point. Hence, $x$ is not a sink. Let $z\in \overrightarrow{F}(x)$. Then $z\in B'$, for some $B'$ such that $B\subseteq B'$. But then $z\in \overleftarrow{F}(A')$. Then $(A, B)$ $\leq$ $\phi((A', B'))$ but not $\phi((A', B'))$ $\leq$ $(A, B)$. This also contradicts that $(A, B)$ is a maximal fixed point.\\
(Only if) Let $A$ be a maximal (w.r.t. $\subseteq$) local conglomerate. We have to prove that $\phi((A, \overrightarrow{F}(A)\cup \overleftarrow{F}(A)))$ $=$ $(A, \overrightarrow{F}(A)\cup \overleftarrow{F}(A))$. This follows from: (i) $A=\{sinks(A)\cup\{x:  \emptyset\neq \overrightarrow{F}(x)\subseteq \overrightarrow{F}(A)\cup \overleftarrow{F}(A)\}$. For this, just observe that since $A$ is a maximal local conglomerate, if $x$ is not a sink then $x$ is such that $\emptyset\neq \overrightarrow{F}(x)\subseteq \overrightarrow{F}(A)\cup \overleftarrow{F}(A)$ iff $x\in A$. 
(ii) $\overrightarrow{F}(A)\cup \overleftarrow{F}(A)$ $=$ $sinks(\overrightarrow{F}(A)\cup \overleftarrow{F}(A))\cup\{x: \emptyset\neq \overrightarrow{F}(x)\cap A\}$. To prove this, note that since $\{x: \emptyset\neq \overrightarrow{F}(x)\cap A\}$ $=$ $\overleftarrow{F}(A)$, it is obvious that $sinks(\overrightarrow{F}(A)\cup \overleftarrow{F}(A))\cup\{x: \emptyset\neq \overrightarrow{F}(x)\cap A\}$ $\subseteq$ $\overrightarrow{F}(A)\cup \overleftarrow{F}(A)$. Now, since $A$ is a local conglomerate, $\overrightarrow{F}(A)\setminus sinks(\overrightarrow{F}(A))\subseteq \overleftarrow{F}(A)$, hence we have that $\overrightarrow{F}(A)\cup \overleftarrow{F}(A)$ $\subseteq$ $sinks(\overrightarrow{F}(A)\cup \overleftarrow{F}(A))\cup\{x: \emptyset\neq \overrightarrow{F}(x)\cap A\}$.

\end{proof}

The class of all the admissible sets of an argumentation framework form a complete partial order with respect to set inclusion (i.e. reflexive, transitive, antisymmetric, and every increasing sequence has a least upper bound). The ``fundamental lemma'' from which the result follows immediately in \cite{dung1995} can be paraphrased in our framework with the help of the operator $\phi$.

\begin{lemma}
Let $\phi((A,$ $\overrightarrow{F}(A)\cup \overleftarrow{F}(A)))$ $=$ $((S^{+}, S^{-}))$ and $A$ be a local conglomerate. Then for every $x, z\in S^{+}$, the set $A'=A\cup\{x\}$ is such that
\begin{enumerate}
  \item $A'$ is a local conglomerate.
  \item Let $\phi((A',$ $\overrightarrow{F}(A')\cup \overleftarrow{F}(A')))$ $=$ $(S'^{+},$ $S'^{-})$. Then $z\in S'^{+}$.
 \end{enumerate}
\end{lemma}

\begin{proof}
\begin{enumerate}
  \item We only have to prove that $\overleftarrow{F}(A')\subseteq S\setminus A'$. Assume the contrary. Since $\overleftarrow{F}(A)\subseteq S\setminus A$, there exists some $w\in A$ such that either $x\in \overrightarrow{F}(w)$ or $w\in \overrightarrow{F}(x)$. Assume $x\in\overrightarrow{F}(w)$. Since $x\in S^{+}$, by definition of $\phi$, $\emptyset\neq \overrightarrow{F}(x)\subseteq\overrightarrow{F}(A)\cup \overleftarrow{F}(A)$. Hence $\overrightarrow{F}(x)\cap A=\emptyset$, which implies that $A$ does not absorb $w$. This contradicts that $A$ is a local conglomerate. Assume now that $w\in \overrightarrow{F}(x)$. Then $x\not\in S^{+}$. Contradiction.
  \item Since $z\in S^{+}$, we have that $\overrightarrow{F}(z)\subseteq\overrightarrow{F}(A)\cup\overleftarrow{F}(A)$. Then, given that $A\subseteq A'$, $\overrightarrow{F}(z)\subseteq\overrightarrow{F}(A')\cup\overleftarrow{F}(A')$. Therefore, $z\in S'^{+}$.
\end{enumerate}
\end{proof}

Caminada \cite{caminada2006} showed that preferred extensions correspond to labellings that maximize $\tt T$ (or, equivalently, $\tt F$). We can consequently put maximal local conglomerates also in correspondence with such labellings in our framework.

\begin{theorem}\label{theo4}
$L$ is a labelling that maximizes $\tt T$  iff $A=\{x: L(x)=\tt{T}\}$ is a maximal local conglomerate.
\end{theorem}

\begin{proof}
(If) Let $A=\{x: L(x)=\tt{T}\}$ be a maximal local conglomerate. Let us assume, by way of the absurd, that there exists $z\not\in A$ such that $L(z)=\tt{T}$. Then for all $w\in \overrightarrow{F}(z)$, $L(w)=\tt{F}$. Then, since $\overrightarrow{F}(A)\setminus sinks(S)\subseteq \overleftarrow{F}(A)$ and $L(x)=\tt{T}$ for every $x\in A$, it follows that $z\not\in \overrightarrow{F}(A)$. Moreover, $z\not\in \overleftarrow{F}(A)$ either, otherwise we would have $L(z)=\tt F$. Then, $A\cup\{z\}$ is a local conglomerate greater than $A$, contradicting the hypothesis. Therefore, for every $z\not\in A$, $L(z)\neq\tt{T}$, which implies that $L$ is a labelling that maximizes $\tt{T}$.\\
(Only if) Let $L$ be a labelling that maximizes $\tt T$. First, it is easy to see that $A=\{x: L(x)=\tt{T}\}$ is a local conglomerate.  Now assume, by way of the absurd, that $A$ is not maximal. Then there exists some local conglomerate $A'$ such that $A\subset A'$. Let $x\in A'\setminus A$. Since $L$ maximizes $\tt T$, then $L(x)\neq\tt T$. If  $L(x)=\tt F$ then there exists some $z\in A$ such that $z\in \overrightarrow{F}(x)$. This contradicts that $A'$ is a local conglomerate. And if  $L(x)=\tt U$ then $x$ is paradoxical, which also contradicts that $A'$ is a local conglomerate. Therefore, $A$ is a maximal local conglomerate.
\end{proof}

So, from the  two previous theorems  we get the following:

\begin{corollary}
For every set of sentences $A\subseteq S$, the following three statements are equivalent:
\begin{enumerate}
  \item $A$ is a maximal local conglomerate.
  \item There exists a labelling $L$ that maximizes $\tt{T}$ and $A$ $=$ $\{x: L(x)=\tt{T}\}$.
  \item  $(A, \overrightarrow{F}(A)\cup \overleftarrow{F}(A))$ is a consistent maximal (w.r.t. $\leq$) fixed point of $\phi$.
\end{enumerate}
\end{corollary}

Dung argued that argumentation frameworks with no stable semantics are not necessarily ``wrong''. Analogously, there is nothing necessarily ``wrong'' in ${\cal F}$-systems having no conglomerates. They may have many meaningful, non paradoxical parts, and those parts can be captured via local conglomerates.

Finally, the notion of local conglomerate clearly extends that of admissible set (provided that $F$ is changed to $F^{-1}$ appropriately). Let us argue why local conglomerates not corresponding to admissible sets can make sense in argumentation frameworks. Consider the system $\langle\{a, b\}, \{(a, b)\}\rangle$. Interpreted as an argumentation framework, we have that  argument $a$ attacks  argument $b$, and the only maximally admissible set is $\{a\}$ (by the way, $\emptyset$ is also admissible, but not maximally). On the other hand, there are two maximal local conglomerates: $\{a\}$ and $\{b\}$. In what sense can $\{b\}$ be ``admissible''?  The notion of admissibility was introduced by Dung as a model of the ``principle'' \emph{``The  one  who  has  the last   word   laughs   best''}.  As such, it is a good model. But we can also think of situations in which the last word is wrong. For example, assume $b$ is an argument without objections except for $a$, which is an argument that everybody in the audience would reject. Then that audience would possibly deem the argument $b$ \emph{in} inasmuch $a$ is \emph{out}, and even if $b$ is the last word. The local conglomerate $\{b\}$ enables to capture that possibility. A similar intuition is also modeled by value-based argumentation frameworks (VAF's) \cite{bench-capon}: if $a$ attacks $b$ but $b$ promotes a value that is preferred to the value promoted by $a$, then $a$ does not \emph{defeat} $b$. For example, consider, on the one hand, an argument that promotes the legalization of abortion as a public health necessity in view of a significant magnitude of women's deaths as a result of clandestine abortion practices and, on the other hand, an argument that promotes the prohibition of abortion as a need for protection of the fundamental human right to the life of an embryo. Suppose the second argument is advanced as an attack on the first. However, if the audience values public health over the right to life of a human embryo, it will result that the expected defeat will have no effect (the same applies in the opposite case). In this setting, we can capture the opposing opinions with two different (local) conglomerates.

\section{Conclusions}\label{conc}

The notion of conglomerate enabled us to essentially capture the same results as that of kernel regarding semantic paradoxes. But while kernels can only represent situations in which all the sink (object language) sentences in the system are true, conglomerates also enable to represent the false cases. Another important result is that conglomerates enable to characterize referential contradictions and tautologies, i.e. sentences that can be assigned only the false value or only the true value, respectively. Referential contradictions cannot belong to any conglomerate. And more generally, as a consequence of Theorem \ref{theo3}, it is easy to see that referential contradictions are excluded from every maximal local conglomerate.

Inspiration for an abstract treatment of ${\cal F}$-systems came from Dung's abstract argumentation frameworks \cite{dung1995}. Caminada's labelling semantics for Dung's systems \cite{caminada2006} were adapted for our purpose here. 
Dyrkolbotn \cite{dyrkolbotn2012} has previously treated paradox, kernel theory and argumentation frameworks on a common ground. He showed the connection between local kernels and admissible sets. We extended here that connection with the notion of local conglomerate, which in turn gives a twist to the notion of admissibility. In that respect, we plan to explore in the future the use of local conglomerates as a semantics for value-based argumentation frameworks \cite{bench-capon}, as suggested in Section \ref{dung}.\\


%
%

\begin{acknowledgements}
The author appreciates the invitation of the Department of Mathematics of the Universidad de Aveiro, Portugal, for a visit in September 2019, during which a previous version of this paper was presented.
\end{acknowledgements}

\bibliographystyle{spmpsci}      
\bibliography{biblio1b}   


%


\end{document}